\documentclass[letterpaper, 10 pt, conference]{ieeeconf}
\IEEEoverridecommandlockouts                              

\let\labelindent\relax
\usepackage{amsmath} 
\usepackage{amssymb}  

\usepackage{ bbold }
\usepackage{tabulary}
\usepackage{hyperref}
\usepackage[english]{babel}
\usepackage{blindtext}
\usepackage[font=footnotesize]{caption}
\usepackage[makeroom]{cancel}
\usepackage{amsfonts}
\usepackage{subfig,pgfplots}
\usepackage{amsthm}  
\usepackage{amssymb}
\usepackage{graphicx}
\usepackage{dsfont}
\usepackage{pgf,tikz}\usepackage{algorithm2e}
\RestyleAlgo{ruled}
\usepackage{algpseudocode}
\pgfplotsset{ 
	compat=newest, 
	legend style =
	{font=\footnotesize },
	label style = {font=\footnotesize},
	every tick label/.append style={font=\footnotesize}
}
\usepackage{tabularx}
\usepackage{float}
\usepackage{cite}
\usepgfplotslibrary{fillbetween}
\usetikzlibrary{shapes.multipart}
\usepackage{enumitem} 
\usepackage{bbm} 

\usepackage[strict]{changepage}

\newcommand{\R}{\mathbb{R}}
\newcommand{\mc}{\mathcal}

\newcommand{\ben}{\begin{equation*}}
	\newcommand{\een}{\end{equation*}}
\newcommand{\be}{\begin{equation}}
	\newcommand{\ee}{\end{equation}}
\newcommand{\ba}{\begin{array}}
	\newcommand{\ea}{\end{array}}

\newtheorem{theorem}{Theorem}
\newtheorem{lemma}{Lemma}
\newtheorem{assumption}{Assumption}

\newtheorem{definition}{Definition}
\newtheorem{corollary}{Corollary}
\newtheorem{proposition}{Proposition}

\newtheorem{problem}{Problem}
\newtheorem{example}{Example}

\allowdisplaybreaks
\title{A Convex Formulation of the Multi-Commodity Dynamic Traffic Assignment}
\author{Davide Sipione, Giacomo Como, Gustav Nilsson
	\thanks{D. Sipione and G. Como are with the Department of Mathematical Sciences, Politecnico di Torino, Torino, Italy (\texttt{\{davide.sipione,giacomo.como\}@ polito.it}). G. Como is also with the Department of Automatic Control, Lund University, Sweden. Gustav Nilsson is with the Univ. Grenoble Alpes, Inria, CNRS, Grenoble INP, GIPSA-Lab, Grenoble, France (\texttt{gustav.nilsson@inria.fr}). This work was partially supported by the Research Project PRIN 2022 ``Extracting Essential Information and Dynamics from Complex Networks'' (Grant Agreement number 2022MBC2EZ) funded by the Italian Ministry of University and Research.}}
\begin{document}
	\maketitle
\begin{abstract}

	We consider a multi-commodity Dynamic Traffic Assignment (DTA) problem formulated as a network flow control problem on the Cell Transmission Model (CTM). The objective is to design optimal control policies using variable speed limits, ramp metering, and dynamic routing to regulate traffic evolution over time on a given limited-capacity transportation network. Even simple instances of DTA problems on the CTM are known to give rise to non-convex optimal control formulations. Nevertheless, a single-commodity DTA formulation has recently been proposed that admits a tight convex relaxation, thereby enabling tractable optimal control synthesis. The single-commodity formulation, however, is structurally restrictive, as it effectively allows only a single destination. To address this limitation, we develop a multi-commodity CTM model in which each commodity is associated with potentially distinct sets of off-ramps. By extending the convexification approach developed for the single-commodity case, we establish a tight convex relaxation of the multi-commodity DTA problem on the CTM model. This relaxation relies on concave, commodity-specific demand functions and concave aggregate supply functions for every cell, which ensure convexity of the resulting optimal control problem. Our proposed formulation requires commodity-dependent implementation of variable speed limits and dynamic routing policies.
    
\end{abstract}

\section{Introduction}
The Dynamic Traffic Assignment (DTA) problem, first introduced in \cite{Merchant1978}, has become a cornerstone of transportation research~\cite{Peeta:2001}. While originally proposed for planning purposes, with the technological development, it has increasingly been leveraged in the context of optimal real-time feedback control~\cite{Gomes:2006}.

A macroscopic, first-order traffic model suitable for addressing the DTA problem is the celebrated Cell Transmission Model (CTM)~\cite{DAG2}. 
Straightforward formulations of the DTA problem with CTM dynamics are known to yield non-convex formulations, which is mainly caused by congestion effects~\cite{Non-Conv}. However, in~\cite{Zili:2000} a linear program is obtained by relaxing the supply and demand constraints of the CTM. Then, in~\cite{Mura2012}, a similar relaxation is found for the link-node cell transmission model by design of variable speed limits and ramp metering, when the demand functions are linear and the supply functions are affine. These results are further extended in~\cite{ConvAndRob} where a tight and convex relaxation of the DTA problem is formulated, for arbitrary network topology, arbitrary concave demand and supply functions, and convex cost functions.

However, none of these studies account for the inherent heterogeneity of traffic flows. As autonomous vehicles become increasingly more common and with the ever-growing traffic demand, a shift in how freeway networks are managed and optimized has become necessary. Several studies have attempted to solve a multi-class/multi-commodity DTA problem. In fact, in~\cite{SO-DTA} the population is split into controllable and selfish agents. Intuitively, the control action can be applied only to the former. The resulting nonlinear optimal control problem is then solved using the discrete adjoint method and a multi-start approach. Similarly, in~\cite{SO-DTA2,SO-DTA3} the population is again split into the sets of compliant and non-compliant users, where the former follow the imposed routing blindly, and the latter employ a logit choice function at diverge junctions. However, these studies lack theoretical optimality guarantees. 

In this paper, we employ a macroscopic model for heterogeneous traffic flows, based on the CTM. We focus on dynamical multi-commodity flow networks, modeled as dynamical systems led by mass conservation laws on directed multigraphs. 
The model is then leveraged to improve traffic performance through optimal control techniques. Specifically, we formulate a Multi-Commodity Dynamic Traffic Assignment (MC-DTA) problem where the control action ---variable speed limits, ramp metering, and dynamical routing--- can be applied separately to every commodity. As mentioned before, straightforward formulations of the MC-DTA problem lead to non-convex instances. To tackle this issue, we extend single-commodity results~\cite{ConvAndRob}, to provide a tight and convex relaxation of the Dynamic Traffic Assignment in a multi-commodity setting, by assuming both concavity of the demand and supply functions and
convexity of the cost function. Hence, we adapt the approach in \cite{FNC}, where the control is actuated only through variable speed limits and ramp metering. In this paper dynamical routing becomes part of the control action and, moreover, we draw insights from Pontryagin's Maximum Principle for the optimality of the proposed solution. Furthermore, it is possible to solve the proposed MC-DTA employing available convex optimization solvers,
such as CVX~\cite{cvx}.

The rest of the paper is organized as follows. The dynamical multi-commodity flow network model is introduced in Section~\ref{sec:model}. In Section~\ref{sec:control}, we formalize the Multi-Commodity Dynamic Traffic Assignment problem and provide a tight convex relaxation. In Section~\ref{sec:example} we present some examples.

\textbf{Notation}
The sets $\mathbb{R}$, $\mathbb{R}_+$, and $\mathbb{R}_{++}$ represent the sets of real, non-negative real, and positive real numbers, respectively. For non-empty finite sets $\mathcal{A}$ and $\mathcal{B}$, $\mathbb{R}^\mathcal{A}$ is the $|\mc A|$-dimensional of real vectors and  $\mathbb{R}^{\mathcal{A} \times \mathcal{B}}$ is the set of $|\mc A|\times|\mc B|$ real matrices, whose entries are indexed by elements of $\mc{A}$ and $\mathcal{A}$ and $\mathcal{B}$, respectively.  

\section{A Controlled Dynamical Multi-Commodity CTM-Based Flow Network Model}\label{sec:model}
In this section, we introduce a macroscopic model for dynamic multi-commodity traffic flows on directed multigraphs, borrowing insights from the celebrated Cell Transmission Model (CTM)~\cite{DAG2}.

\subsection{Multi-Commodity Flow Network Model}
We model a transportation network as a finite directed multigraph $\mc G = (\mc V, \mc E)$, where $\mc V$ and $\mc E$ denote the set of nodes and edges, respectively. From a physical point of view, edges ---referred to as \emph{cells}--- represent sections of roads, while nodes correspond to junctions that connect subsequent cells. 
For every cell $i$ in $\mc{E}$, let $\sigma_i,\tau_i$ in $\mc{V} \cup \{w\}$ denote its tail and head node, respectively. Here $w$ is a virtual node representing the external world. We assume that there are no self-loops, i.e., $\sigma_i \neq \tau_i$ for every cell $i$ in $\mc{E}$.
The sets of on-ramps and off-ramps are denoted by $\mc O=\{i\in\mc E:\sigma_i=w\}$ and $\mc S=\{i\in\mc E:\tau_i=w\}$, respectively. 
Moreover, we denote by $\mc{A} = \left\{(i,j) \in \mc{E} \times \mc{E} : \tau_i = \sigma_j \neq w\right\}$ the set of adjacent pairs of cells. 


The network is shared by a nonempty finite set of commodities~$\mc K$, who may be restricted to different ~subgraphs of $\mc G$. To every commodity $k$ in $\mc{K}$ we associate a set $\mc{E}_k\subseteq\mc E$ of utilizable cells. We assume that, for every commodity $k$ in $\mc K$, the sets of utilizable on-ramps $\mc{O}_k=\mc{O}\cap\mc E_k$ and off-ramps $\mc{S}_k=\mc{S}\cap\mc E_k$ are nonempty and that every utilizable cell $i$ in $\mc{E}_k$ belongs to a path from some on-ramp $o$ in $\mc{O}_k$ to some off-ramp $r$ in $\mc{S}_k$.
To each commodity is associated its own set of adjacent pairs of cells $\mc{A}_k$, i.e., $\mc{A}_k = \left\{(i,j) \in \mc{E}_k \times \mc{E}_k: \tau_i = \sigma_j \neq w\right\}$. 

%
The state space of the system is 
$$\mc X:=\left\{x\in\R_+^{\mc E \times \mc K}:\,x_{ik}=0\;,\; \forall i\in\mc E\setminus\mc E_k\right\}\,.$$
Every entry $x_{i}^{(k)}$ of a state $x$ in $\mc X$ represents the traffic volume of commodity $k$ on cell $i$. 
Borrowing from the fundamental diagram of traffic flows, two families of functions are introduced. The maximum possible outflow $d_i^{(k)}(x_i^{(k)})$ from a cell $i$ depends on the traffic volume $x_i^{(k)}$ through a demand function $d_i^{(k)} : \R_+ \rightarrow \R_+$. The maximum possible inflow $s_i(\sum_{k \in \mc{K}}{w_i^{(k)}x_i^{(k)}})$ into a cell $i$ depends on a weighted sum of all traffic volumes in cell $i$ through a shared supply function $s_i: \R_+ \rightarrow \R_+$.
Note that $w_i^{(k)} > 0$ is a positive weight for the traffic volume in cell $i$ that is allowed to be dependent on the commodity, in case different commodities will occupy different amount of space in the cell. 
\begin{assumption}\label{ass:1}
	\begin{enumerate}
		\item[(i)] The supply function of every non-onramp cell $i$ in $\mc E \setminus \mc O$ is continuously differentiable, concave, and non-increasing.
        
	\item[(ii)] The demand function of every cell $i$ in $\mc{E}_k$ and commodity $k$ in $\mc{K}$ is continuously differentiable, concave, non-decreasing, and such that $d_i^{(k)}(0) = 0$ for every $k$ in $\mc{K}$ and $i$ in $\mc{E}_k$.
	\end{enumerate}
\end{assumption}

\begin{example}\label{ex:1}
Linear demand and affine supply functions $$d_i^{(k)}(\xi) = \min \left\{\frac{v_i^{(k)}}{L_i}\xi, q_i^{(k)}\right\}\,,\quad s_i(\xi) = \max\{0,C_i - w_i\xi\}\,,$$ where $v_i \geq 0$ is the speed at which flow travels through the cell, $L_i > 0$ is the length of the cell, $q_i^{(k)} > 0$ is the maximum out-flow of type-k flow from the cell, $C_i$ is the capacity of the cell, and $w_i >0$ the congestion wave speed. 
\end{example}

\subsection{Controlled FIFO Dynamical Multi-Commodity Flow Networks}

The flow entering the network is modeled as an exogenous time-varying inflow $\lambda_i^{(k)}(t) \geq 0$ at each on-ramp $i$ in $\mc O_k$. For every cell $i$ in $\mc E_k \setminus \mc O_k$ the exogenous inflow is identically set to 0, i.e., $\lambda_i^{(k)}(t) \equiv 0$. We denote by $\lambda^{(k)}(t)=(\lambda^{(k)}_i(t))_{i\in\mc E}$ the vector of exogenous inflows of commodity $k$ across all cells.

In this setting, control is actuated both through dynamical routing and variable speed limits combined with ramp metering. The former controls how type-$k$ flow splits at diverge junctions, while the latter limits the rate at which flow leaves a cell. 

Intuitively, routing must satisfy some physical constraints. To this end, let us call $M^{(k)}$ the routing matrix of ~commodity $k$, whose entries are $0 \leq M_{ij}^{(k)} \leq 1$, called \emph{turning ratios}, denote the fraction of flow of commodity $k$ routed from
~cell $i$ to cell $j$. Hence, every routing matrix $M^{(k)}$ must satisfy the following constraints:

\be\label{eq:sub}
\sum\limits_{j \in \mc{E}}{M_{ij}^{(k)}} = \begin{cases}
	1 & i \in \mc{E}_k \setminus \mc{S}_k \\
	0 & i \in \mc{S}_k
\end{cases}\; ;
\ee
\be\label{eq:zero}
M_{ij}^{(k)} = 0, \quad \forall (i,j) \in \mc{E} \times \mc{E} \setminus \mc{A}_k \;.
\ee

Equation \eqref{eq:sub} ensures that only off-ramps are allowed to send flow to the external world, while Equation \eqref{eq:zero} implies that flow can only travel through adjacent cells.

Hence, for every commodity $k$ in $\mc{K}$ we can define the set of routing matrices that satisfy these constraints as $\mc{R}^{(k)} = \left\{M^{(k)} \in \R_+^{\mc{E} \times \mc{E}}: \eqref{eq:sub},\eqref{eq:zero}\right\}$.
Moreover, we define the set of admissible variable speed limits and ramp metering as
$\Omega = [0,1]^{\mc{E}\times\mc{K}}$, 
since they limit the outflow of every commodity and cell. Before proceeding notice that the set $\Omega$ and $\mc{R}^{(k)}$ are both compact and convex.

In the following definition, we formally introduce the control actions. 
 
\begin{definition}
	The control actions considered are
	\begin{enumerate}
		\item[(i)] Dynamical routing, which is modeled through a family of piecewise continuous functions  $R^{(k)}: [0,T] \rightarrow \mc{R}^{(k)}$ for all $k \in \mc{K}$;
	
		\item[(ii)] Variable speed limits and ramp metering, which are also modeled through a piecewise continuous function $\alpha : [0,T] \rightarrow \Omega$.
	\end{enumerate}
\end{definition}
 
In our setting, variable speed limits and ramp metering decrease the maximum allowed outflow of commodity $k$ from cell $i$ from $d_i^{(k)}(x_i^{(k)})$ to $\alpha_i^{(k)}d_i^{(k)}(x_i^{(k)})$. It is important to note that this allows distinct control parameters to be assigned to each commodity, enabling separate control. 

The dynamical systems is driven by mass conservation laws, stating that the variation of commodity $k$ in $\mc{K}$ traffic volume in a cell $i$ in $\mc{E}$ is equal to the difference between the total inflow to and the total outflow from that cell, i.e.,
\be\label{eq:massCons}
	\dot{x}_{i}^{(k)}  =   \lambda_{i}^{(k)} + \sum\limits_{j \in \mathcal{E}^{(k)}}{f_{ji}^{(k)}(x)}  - z_{i}^{(k)}(x) \,,
\ee
where $f_{ji}^{(k)}(x)$ denotes the flow sent from cell $j$ to cell $i$ and $z_i^{(k)}(x)$ denotes the total outflow from cell $i$ as
\be\label{eq:zik}
		z_{i}^{(k)}(x)  =  \begin{cases}
		\sum\limits_{j \in \mathcal{E}^{(k)}}{f_{ij}^{(k)}(x)} & i \in \mathcal{E}^{(k)} \setminus \mathcal{S}^{(k)} \\
		\alpha_i^{(k)}d_{i}^{(k)}(x_{i}^{(k)}) & i \in \mathcal{S}^{(k)}\,,
	\end{cases}
\ee
both expressed as functions of the current state $x$. Let $\dot{x}^{(k)} = (\dot{x}_i^{(k)})_{i \in \mc{E}}$  and $ \dot{x} = (\dot{x}_i^{(k)})_{i \in \mc{E},k \in \mc{K}}$ denote the vector of commodity $k$ traffic variation and the matrix of traffic variation, respectively. 

The only step left is to model the relationship between the cell-to-cell flow $f_{ij}^{(k)}$ and the state $x$. In spite of this, we employ a First-In-First-Out (FIFO) allocation rule. Formally,
	\begin{equation}\label{eq:fifo}
		f_{ij}^{(k)}(x^{(k)}) = \gamma_i(x,\alpha,R)R_{ij}^{(k)}\alpha_i^{(k)}d_{i}^{(k)}(x_{i}^{(k)}) \; ,
	\end{equation}
	where $\gamma_{i} =\max \{  \tau \in [0,1] :    \eqref{eq:4.5}\}$, 
	\begin{multline}\label{eq:4.5}
		\tau \cdot \max_{j \in \mc E \,, R_{ij}^{(k)} >0} \sum_{k \in \mc K} \biggl (\sum_{h\in \mc E_k} R_{hj}^{(k)}\alpha_h^{(k)} d_h^{(k)}(x_h^{(k)}) \biggr)\\  \leq s_j\biggl(\sum\limits_{k \in \mc{K}} w_j^{(k)}x_j^{(k)}\biggr) \,.
	\end{multline}

Equation \eqref{eq:4.5} implies that in a diverge junction, if for a commodity $k$ in $\mc{K}$ a downstream cell $i$ in $\mc{E}_k$ is congested then the flow directed to every downstream cells in $\mc{E}_k$ is limited accordingly. Moreover, in ordinary junctions, the flow is split proportionally among commodities that share that cell when congestion downstream occurs.

Given these considerations, it is now possible to formally state the Controlled Dynamical Multi-Commodity Flow Network model. 
\begin{definition}
	Given a nonempty finite directed multigraph $\mc{G} = (\mc V, \mc E)$, supply functions $s_i$ and demand functions $d_i^{(k)}$ satisfying Assumption \ref{ass:1}, exogenous inflow array $\lambda = (\lambda_i^{(k)})_{i \in \mc{E},k \in \mc{K}}$, and control actions $R^{(k)}(t)$ and $\alpha(t)$, a Dynamical Multi-Commodity Flow Network (DMFN) is a dynamical system satisfying equations \eqref{eq:massCons}--\eqref{eq:4.5}.
\end{definition}

We then establish the well-posedness of the initial value problem, as already proven in \cite{Stab_D}, in the special case when $\alpha_i^{(k)} \equiv 1$.
\begin{lemma}\label{lemma:well_pos}
	For every initial condition $x(0)$ in $\mc{X}$, there exists a unique solution $x(t)$ of the DMFN at time $t \geq 0$.
\end{lemma}
\begin{proof}
	First, let $\dot{x} = g(t,x)$ and notice that $g(t,x)$ is piecewise continuous in $t$ (since the functions $R^{(k)}$ and $\alpha$ are piecewise continuous) and globally Lipschitz continuous in $x$, according to Assumption~\ref{ass:1}. Hence, from Picard-Lindelöf theorem~\cite{Khalil} it follows that every initial condition $x(0)$ there exists a unique solution of the DMFN at every time $t \geq 0$.
\end{proof}
\subsection{Optimal Multi-Commodity Dynamic Traffic Assignment}
In this section we are interested into optimizing the time-integral of a cost function $\varphi(x,z)$, which depends on the traffic volume $x$ and cell outflow $z$, over a finite time interval $[0,T]$, where $T$ denotes the time horizon. The following standard assumption is made:
\begin{assumption}\label{ass:cost_fleet}
	The cost function $\varphi(x,z)$ is convex and differentiable in $(x,z)$ and such that
	$\frac{\partial}{\partial x}\varphi\left(x,z\right) \geq 0$ , $\frac{\partial}{\partial z}\varphi\left(x,z\right) \leq 0$ , and
	$\varphi(0,0) = 0$, i.e., non-decreasing in $x$, non-increasing in $z$ and equal to $0$ when $(x,z) = (0,0)$.
\end{assumption}

With the introduction of the cost function, it is now possible to formally state the optimal control problem under investigation. 


\begin{problem}[Continuous time MC-DTA]\label{p:c}
	\be\label{eq:part_DTA}
	\ba{rcl}
	J(\alpha^*,R^*) = &\underset{\alpha(t) \in \Omega,R^{(k)}(t) \in \mc{R}^{(k)}}{\emph{minimize}}  & \int_0^T{\varphi(x(t),z(t))dt} \\
&\emph{subject to}  & \eqref{eq:massCons},\eqref{eq:zik}, \eqref{eq:fifo}, \eqref{eq:4.5}
	\ea
	\ee
\end{problem}

Note that this optimal control problem yields a non-convex formulation, mainly given by congestion handling at diverge junctions~\cite{Non-Conv}. Moreover, we will let $J(\alpha,R)$ denote the cost of a feasible solution of Problem~\ref{p:c}.

As a standard result, we prove the existence of the optimal control for Problem~\ref{p:c}.

\begin{proposition}\label{prop:filippov}
	Given a cost function satisfying Assumption~\ref{ass:cost_fleet}, for any initial condition $x^0$ in $\mc{X}$ the Problem~\ref{p:c} admits an optimal solution.
\end{proposition}
\begin{proof}
	Let $\dot{x} = g(x,\alpha,R)$ and let $\tilde{x} = \int_0^T \varphi(x(t),z(t))dt$, so that the augmented dynamics read
	$$
		\dot{\hat{x}} = \begin{pmatrix}
			\dot{x} \\ \dot{\tilde{x}}
		\end{pmatrix} = \begin{pmatrix}
		g(t,x) \\ \varphi(x(t),z(t))
		\end{pmatrix}\;.
	$$
	We are then left to prove that the assumptions of Filippov's theorem~\cite{Liberzon} hold for our augmented system. First, as previously noted the sets $\Omega$ and $\mc{R}^{(k)}$ are compact, so that we are left to prove that the set
	$$
	G(x) = \left\{\begin{pmatrix}
		g(x,\alpha,R) \\
		\varphi(x(t),z(t))
		
	\end{pmatrix} \Bigg|\; \alpha \in \Omega\;, R^{(k)} \in \mc{R}^{k}\right\}
	$$
	is convex for every $x$ in $\mc{X}$. Since $g(x,\alpha,R)$ is affine in $\alpha$ and $R$ and the sets $\Omega$ and $\mc{R}$ are convex, this is immediate. Finally, the existence of a solution for every $t > 0$ is verified by Lemma~\ref{lemma:well_pos}. 
	Hence, the reachable set from $(x^0,\tilde{x}^0)$ is compact and thus admits a minimum.
\end{proof}

\section{Tight Convexification of the MC-DTA problem}\label{sec:control}
To tackle the non-convexity of the optimal control problem, we shall now introduce a relaxed formulation and proceed to prove that it is both convex and tight, as similarly done for the Freeway Network Control in \cite{FNC}. The key insight behind this relaxation is to shift the optimization variables from the control parameters $\alpha_i^{(k)}(t)$ and the turning ratios $R_{ij}^{(k)}(t)$ to the traffic volume variables $ \bar {x}^{(k)}(t) $ in $\R_+^{\mc E_k}$ and the traffic flow variables $\bar{f}^{(k)}(t)$ in $\R_+^{\mc A_k}$ and $\mu_i^{(k)}$ in $\R_+^{\mc{E}_k}$, such that $\mu_i^{(k)} = 0$ for every $i$ in $\mc{E}_k\setminus\mc{S}_k$. Regarding the traffic flow variables, $\bar{f}$ is the flow between cells within the network, while $\mu$ captures the flow that leaves the network through the off-ramps

The relaxation should be consistent our model, which is inspired from the fundamental diagram of traffic flows, so that these new optimization variables must satisfy the demand and supply constraints, i.e.,
\begin{enumerate}
	\item[(i)] demand constraints
	\be\label{eq:dmnd}
		\mu_i^{(k)}(t) + \sum\limits_{j \in \mc{E}_k}\bar{f}_{ij}^{(k)}(t) \leq d_i^{(k)}(\bar x_i^{(k)}(t))\;,
	\ee
	\item[(ii)] supply constraints
	\be\label{eq:supp}
		\sum\limits_{k \in \mc{K}}\sum\limits_{j \in \mc{E}_k}\bar f_{ji}^{(k)}(t) \leq s_i\biggl(\sum\limits_{k \in \mc{K}} w_i^{(k)}\bar{x}_i^{(k)}(t)\biggr)\;,
	\ee
\end{enumerate}
Let $\bar{u}(t)$ be the concatenation of vectors $ (\bar{f}(t),\mu(t))$, then it is possible to write the demand and supply jointly in compact vector form as $D\bar{u}(t) \leq h(\bar{x}(t))$, 
where $D \in \R_+^{r\times m}$ with $r = \sum_{k \in \mc{K}}|\mc{E}_k| + |\mc{E}|$ and $m = \sum_{k \in \mc{K}} |\mc{A}_k| + |\mc{S}_k|$. 

Once the new variables and constraints are introduced, we are ready to state the relaxation of the MC-DTA problem:
\begin{problem}[Relaxation of the continuous time MC-DTA]\label{p:c_r}
	\be\label{eq:DTA_R}
	 \ba{rcl}
	\bar{J}(\bar{x}^*,\bar{u}^*) = \underset{\bar{x}(t),\bar{u}(t)}{\emph{minimize}} \hspace*{-5em} & &\hspace*{4em} \int_0^T{\varphi(\bar{x},\bar{u})\;dt}\\
\emph{s.t. }  &  &\hspace{-1.5em}   \dot{\bar{x}}_i^{(k)} = \lambda_i^{(k)}(t) + \sum\limits_{j \in \mc{E}_k}\bar f_{ji}^{(k)}(t) - \sum\limits_{j\in\mc E_k}\bar{f}_{ij}^{(k)}(t) \\
& &\hspace{1.5em}- \mu_i^{(k)}(t)\;,\\
	& &\hspace{-1em}    D\bar{u}(t) - h(\bar{x}(t)) \leq 0, \\
	& & \hspace{-1em} \mu_i^{(k)}(t) = 0\;,\; \forall i \not\in \mc{S}\\
	& &\hspace{-1em} \bar x^{(k)}(0) = \bar x^{(k),0}\;,\\

	& &\hspace{-1em}  \forall k \in \mc{K}\;, \forall i \in \mc{E}_k\;.
	\ea\ee
\end{problem}

Similarly, we shall call $\bar{J}(\bar{x},\bar{u})$ the cost achieved by a feasible solution of Problem~\ref{p:c_r}. In the following proposition, we shall then prove that Problem \ref{p:c_r} is a convex relaxation of Problem~\ref{p:c}.

\begin{proposition}\label{prop:conv_relax}
   Problem \ref{p:c_r} is convex, and every feasible solution of Problem \ref{p:c} is also feasible for Problem \ref{p:c_r}.
\end{proposition}
\begin{proof}
	Firstly, notice that for every choice of control parameters $\alpha_i^{(k)}(t)$ in $[0,1]$ and turning ratios $R_{ij}^{(k)}(t)$ the demand \eqref{eq:dmnd} and supply constraints \eqref{eq:supp} are inevitably satisfied. Moreover, every feasible solution of \eqref{eq:part_DTA} is also feasible for the relaxed problem \eqref{eq:DTA_R}.
	Then, we are left with proving the convexity of \eqref{eq:DTA_R}. Notice that this would mean that if both
	$\left( \bar x^0(t), \allowbreak\bar u^0(t)\right)$ and $\left( \bar x^1(t),\bar u^1(t)\right)$ satisfy the constraints, then for every $\beta$ in $\left[0,1\right]$, also $\left( \bar x^\beta(t),\bar u^\beta(t)\right)$ does, where $\bar x^\beta = (1-\beta)\bar x^0 + \beta \bar x^1$ and $\bar u^\beta = (1-\beta)\bar u^0 + \beta \bar u^1$. This implies that
    \begin{equation}\label{eq:cond}
        	\int_0^T{\varphi(\bar x^\beta(t))dt} \leq  (1-\beta)\int_0^T{\varphi(\bar x^0(t))dt} + \beta\int_0^T{\varphi(\bar x^1	(t))dt}
    \end{equation}
	Since $\varphi(\bar x,\bar u)$ is convex by assumption, then inequality \eqref{eq:cond} is always satisfied. Moreover, the concavity of the demand and supply functions implies the convexity of the ~constraints $\eqref{eq:dmnd}$ and $\eqref{eq:supp}$, while constraint \eqref{eq:zik} is a linear equality.
\end{proof}
It remains to establish the tightness of the above convex relaxation.
In particular, we will prove that for every solution of the relaxed convex optimal control problem \eqref{eq:DTA_R}, there exists a corresponding selection of control parameters $\alpha_i^{(k)}$ and turning ratios $R_{ij}^{(k)}$ such that equations \eqref{eq:massCons}, \eqref{eq:zik}, \eqref{eq:fifo}, and~\eqref{eq:4.5} are satisfied. To this end, consider a solution $(\bar x,\bar u)$  and let the control parameters $\alpha^{(k)}$ and the dynamical routing $R^{(k)}$ be given by
\be\label{eq:alfa_sol}
\alpha_i^{(k)}(t) = \frac{\mu_i^{(k)}(t)+\sum_{j \in \mc{E}_k}\bar{f}_{ij}^{(k)}(t)}{d_i^{(k)}(\bar x_i^{(k)}(t))}, \qquad \forall i \in \mc{E}_k\;,
\ee
\be\label{eq:R_sol}
R_{ij}^{(k)}(t) = \begin{cases}
	\frac{\bar f_{ij}^{(k)}(t)}{\sum_{j \in \mc{E}}\bar{f}_{ij}^{(k)}(t)} & \forall (i,j) \in \mc{A}_k\\
	0 & \forall (i,j) \in \mc{E} \times \mc{E} \setminus \mc{A}_k
\end{cases}
\;,
\ee
with the convention that $\alpha_i^{(k)} = 1$ if $\sum_{j \in \mc{E}}\bar{f}_{ij}^{(k)} = d_i^{(k)}(\bar x_i^{(k)}) = 0$. If $\sum_{j \in \mc{E}}\bar f_{ij}^{(k)} = 0$, then the turning ratios are supposed uniform over the set of outgoing cells, i.e. $R_{ij}^{(k)} = |\left\{l \in \mc{E}_k:(i,l) \in \mc{A}_k\right\}|^{-1}$.

\begin{proposition}\label{prop:relax}
	For every feasible solution $(\bar x(t),\bar u(t))$  of the convex MC-DTA \eqref{p:c_r}, it is possible to set the entries of $\alpha(t)$ and $R(t)$ as defined in \eqref{eq:alfa_sol} and \eqref{eq:R_sol}, respectively. Then,
	\begin{enumerate}
		\item[(i)] $(\alpha,R)$ is a solution of problem \eqref{p:c};
		\item[(ii)] $J(\alpha,R) \leq \bar{J}(\bar x,\bar u)$.
	\end{enumerate}
\end{proposition}
\begin{proof}
		Let $(\bar x(t),\bar u(t))$ be a feasible solution. It follows from the demand constraints~\eqref{eq:dmnd} that the outflow is less than or equal to the demand, i.e.,  $\mu_i^{(k)}+\sum_{j \in \mc{E}}\bar f_{ij}^{(k)} \leq \alpha_i^{(k)}d_i^{(k)}(\bar x_i^{(k)})$. Hence, when selecting the control parameters as in \eqref{eq:alfa_sol}, the outflow of every cell corresponds exactly to the demand. Formally,
	$\mu_i^{(k)}+\sum_{j \in \mc{E}}\bar f_{ij}^{(k)} = \alpha_i^{(k)}d_i^{(k)}(\bar x_i^{(k)})$ for all $ i \in \mc{E}_k$. Moreover, for every non off-ramp cell, it follows that $\bar f_{ij}^{(k)} = R_{ij}^{(k)}\sum_{j \in \mc{E}}\bar f_{ij}^{(k)} = R_{ij}^{(k)}\alpha_i^{(k)}d_i^{(k)}(\bar x_i^{(k)})$, for all $ (i,j) \in \mc{A}_k$,
	which implies that for every cell $j$ in $\mc{E}_k \setminus \mc{S}_k$,
	\ben
	\sum\limits_{k \in \mc{K}}\sum\limits_{j \in \mc{E}_k} {R_{ij}^{(k)}\alpha_i^{(k)}d_i^{(k)}(\bar x_i^{(k)})} \leq s_j \biggl(\sum\limits_{k \in \mc{K}}w_j^{(k)}\bar{x}_j^{(k)}\biggr)
	\een
	The last inequality follows from the supply constraint \eqref{eq:supp}, which implies that there is no congestion, i.e., $\gamma_i^{F} = 1$, for every $k$ in $\mc{K}$ and $i$ in $\mc{E}_k$. After that it is easy to verify that the choice of routing matrices $R^{(k)}(t)$ is in $\mc{R}^{(k)}$ for every $k$ and $t$. Hence, $x = \bar{x}$, $f = \bar{f}$ for every cell, and $z = \mu$ for every off-ramp cell. Moreover, the equations $\eqref{eq:massCons},\eqref{eq:zik},\eqref{eq:fifo},\eqref{eq:4.5}$, are satisfied so that $(\alpha,R)$ is a solution of the original Problem~\ref{p:c}. To prove (ii), notice that setting $\alpha$ and $R$ as in $\eqref{eq:alfa_sol},\eqref{eq:R_sol}$ induces the same values of traffic volumes and flows to the original problem, which implies $J(\alpha,R) \leq \bar{J}(\bar{x},\bar{u})$.
\end{proof}

\begin{theorem}
	Suppose the cost function satisfies Assumption~\ref{ass:cost_fleet} and for every initial condition $x^0$ in $\mc{X}$ set $\alpha^*$ and $R^*$ as in \eqref{eq:alfa_sol} and \eqref{eq:R_sol}, respectively. Then the following statements are equivalent.
	\begin{enumerate}
		\item[(i)] $(\alpha^*,R^*)$ is an optimal solution of Problem \ref{p:c};
		\item[(ii)] $(\bar x^*,\bar u^*)$ is an optimal solution of Problem \ref{p:c_r};
		\item[(iii)] there exists $\chi$, ~$\eta : \R_+ \rightarrow \R_+^r$ such that $(\bar x^*,\bar u^*)$ satisfy the PMP conditions
	\ben
	\ba{rcl}
	H(x,u,\chi) & = & -\varphi(x,z) + \langle \chi,\dot{x} \rangle\\
	\eta^\top\left(Du^* - h(x^*)\right) & = & 0\\
	\dot{\chi}_i^{(k)} & = & -\frac{\partial \varphi(x^*,u^*)}{\partial x_i^{(k)}}  - \frac{\partial h(x^*)}{\partial x_i^{(k)}}^\top \eta \\
	\frac{\partial}{\partial f}H(x^*,u^*,\chi) & = & D^\top\eta \\
	\max\limits_{f} H(x^*,u,\chi) & =& H(x^*,u^*,\chi)\\
	\chi_i^{(k)}(t) & \neq & 0\;, \forall 0 \leq t < T\\
	\chi_i^{(k)}(T) &  = & 0
	\ea
	\een
	\end{enumerate}
\end{theorem}
\begin{proof}
    We first focus on proving that (i) and (ii) are equivalent. Let $(\alpha^*,R^*)$ be the optimal control solution for the original problem. By Proposition \ref{prop:conv_relax} this is also feasible for the relaxed problem, such that $\inf \bar{J}(\bar x,\bar u) \leq J(\alpha^*,R^*)$.  Then from Proposition \ref{prop:relax} we know that we can map any solution $(\bar x, \bar u)$ of the relaxed problem to a solution $(\alpha,R)$ of the original one, such that $J(\alpha,R) \leq \bar{J}(\bar x,\bar u)$. By taking the infimum on both sides we get $J(\alpha^*,R^*) \leq \inf \bar{J}(\bar x,\bar u)$, which combined with the previous inequality gives the identity $J(\alpha^*,R^*) = \inf J(\bar x,\bar u) = J(\bar x^*,\bar u^*)$. Hence, the relaxed problem also admits an optimal solution $(\bar x^*,\bar u^*)$.
	
	We are now left to prove that (ii) is equivalent to (iii). This reduces to proving that the PMP conditions are both necessary and sufficient for optimality. As far as necessity is concerned, since the demand and supply functions are $\mc{C}^1$ we easily satisfy the continuity assumptions of the Pontryagin Maximum Principle with mixed constraints~\cite{Aruty2010}. However, we still need to verify regularity condition in \cite[Definition 2]{Aruty2010}. Hence, let $\mc{B}_a^n$ be the ball of dimension $n$ and radius $a$ with center at the origin, we need to prove that there exists $\delta > 0$ such that  $\mc{B}_\delta^r \subseteq D\mc{B}_1^m + Du^* - h(x^*) + \R_+^r$. To verify this, it is sufficient to prove that $\text{Im}(D) \cap \R_{++}^r \neq \emptyset$. Indeed, if this is the case then $\exists v \in (\R_{++}^r \cap D\mc{B}_1^m)$, so that it is sufficient to take $\delta = \min_i v_i > 0$. When $D \in \R_+^{r \times m}$, as is our case, it is sufficient to verify that $D\mathbb{1} > 0$, namely that every constraint involves at least one control variable, which is true in our setting.
	To prove sufficiency we exploit the convexity of our problem. Given this property, sufficiency follows easily by noting that the PMP conditions are the same provided in \cite[Corollary 1.(c)]{Manga1996}.
\end{proof}

\section{Examples}\label{sec:example}
\subsection{An Analytical Example}
In this section we shall then exploit the PMP conditions just proposed in a simple setting to understand how the optimal controller behaves. To this end, consider the diverge junction composed by cells $1$, $2$, and $3$ in Figure \ref{fig:multi_od}, and assume that in this example cells $2$ and $3$ are off-ramps.

\begin{figure}
	\centering

	\begin{tikzpicture}[thick]
		
		\node[draw,circle] (a)  at (0,0) {};
		\node[draw,circle] (b) at (1.5,0) {};
		\node[draw,circle] (c) at (3,0) {};
		\node[draw,circle] (d)  at (0,-1) {};
		\node[draw,circle] (e) at (3,-1) {};
		\node[draw,circle] (f) at (4.5,-1) {};
		
		\draw[->] (-1,0) --node[above]{$1$} (a);
		\draw[->] (a) --node[above] {$2$} (b) ;
		\draw[->] (b) --node[above] {$6$} (c) ;
		\draw[->] (c) --node[above] {$10$} (4.5,0);
		\draw[->] (a) --node[left] {$3$} (d);        
		\draw[->] (-1,-1) --node[below]{$4$} (d);
		\draw[->] (d) --node[below]{$5$} (e);
		\draw[->] (e) --node[below]{$8$} (f);
		\draw[->] (e) --node[right]{$7$} (c);
		\draw[->] (f) --node[below]{$9$} (5.5,-1);
		
	\end{tikzpicture}

	\caption{Synthetic network 2 on-ramps and 2 off-ramps}

	\label{fig:multi_od}

\end{figure}
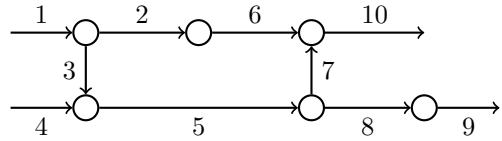

    
The network is shared by two commodities, namely $a$ ~and $b$, but commodity $b$ is restricted to cells $1$ and $2$, while commodity $a$ can travel freely through the network. For simplicity, we shall consider linear demand functions, i.e., $d_i^{(k)}(x_i^{(k)}) = x_i^{(k)}$ for every $i$, and affine supply functions, i.e., $s_i(\sum_{k \in \mc{K}}x_i^{(k)}) = \gamma- x_i^{(a)} - x_i^{(b)}$. 
Let $\chi \in \R^5$ denote the adjoint state and focus on maximizing the following Hamiltonian, with objective function $\varphi(x) = \sum_{i \in \mc{E}} x_i^{(a)} +
 x_i^{(b)}$.
\begin{multline*}
    H(x,f,\chi) =  -\sum_{i \in \mc{E}}\bigl(x_i^{(a)} + x_i^{(b)}\bigr) + \chi_1\bigl(-f_{12}^{(a)} - f_{13}^{(a)}\bigr) \\ 
+  \chi_2\bigl( -f_{12}^{(b)}\bigr) + \chi_3\bigl(f_{12}^{(a)} - x_2^{(a)}  \bigr) 
 +  \chi_4\bigl(f_{12}^{(b)} - x_2^{(b)} \bigr) 
 \\ + \chi_5\bigl(f_{13}^{(a)} - x_3^{(a)}\bigr) 
 =  c(x,\chi) + \kappa_{12}^{(a)}f_{12}^{(a)} + \kappa_{12}^{(b)}f_{12}^{(b)} + \kappa_{13}^{(a)}f_{13}^{(a)} 
\end{multline*}
where $\kappa_{12}^{(a)} = \chi_3 - \chi_1$, $\kappa_{12}^{(b)} = \chi_4 - \chi_2$, and $\kappa_{13}^{(a)} = \chi_5 - \chi_1$.

However, this optimization is constrained by the demand and supply functions. Hence, let $\eta : \R_+ \rightarrow \R_+^4$ denote the Lagrange multipliers associated with the constraints, so that
$$
\ba{rcl}
& & \eta_1\bigl(f_{12}^{(a)} + f_{13}^{(a)} - x_1^{(a)}\bigr)=0\;,\; \eta_2\bigl(f_{12}^{(b)} - x_1^{(b)}\bigr)=0 \\
& & \eta_3\bigl(f_{12}^{(a)} + f_{12}^{(b)} - \gamma + x_2^{(a)} + x_2^{(b)}\bigr) = 0 \;, \\
& & \eta_4\bigl(f_{13}^{(a)} - \gamma + x_3^{(a)} \bigr) = 0 \;,
\ea
$$
Moreover, the adjoint dynamics read
\begin{multline*}
    \dot{\chi}_1 = 1 - \eta_1\;, \quad   \dot{\chi}_1 = 1 - \eta_1\;, \quad  \dot{\chi}_3 = 1 + \chi_3 + \eta_3\;, \\ 
    \dot{\chi}_4 = 1 + \chi_4 + \eta_3\;, \quad \dot{\chi}_5 = 1 + \chi_5 + \eta_4\;, 
\end{multline*}
with transversality condition $(\chi(T))_{i = 1}^5 = 0$ and $(\chi_i(t))_{i=1}^5 \neq 0$ for every $0 \leq t < T$. Notice that this implies that $\chi_3<0$, $\chi_4 < 0$, and $\chi_5 <0$, because if they were positive their derivative would also be positive, which would prevent them from satisfying the transversality condition.  Lastly, we consider the stationarity condition, i.e.,
\begin{align*}
\eta_1 + \eta_3 & =  \max \{0,-\chi_1 + \chi_3\} = \max \{0,\kappa_{12}^{(a)}\}\;,\\
 \eta_1 + \eta_4 & =  \max \{0,-\chi_1 + \chi_5\}= \max \{0,\kappa_{13}^{(a)}\}\;,\\
 \eta_2 + \eta_3 & =  \max \{0,-\chi_2 + \chi_4\} = \max \{0,\kappa_{12}^{(b)}\}\;,
\end{align*}
where the $\max$ operator comes from the non-negativity of the Lagrange multipliers. This implies that $\chi_1 < 0$ and $\chi_2 < 0$ since otherwise the multipliers would be identically zero.

We then proceed to study the sign of $\kappa_{12}^{(a)}$. To this end, consider $\dot{\kappa}_{12}^{(a)} = \dot{\chi}_3 - \dot{\chi}_1 = \chi_3 + \eta_1 + \eta_3 \leq \max\{0, \kappa_{12}^{(a)}\}$, where the last inequality follows from the negativity of $\chi_3$ and the condition of stationarity. Combining this with $\kappa_{12}^{(a)}(T) = 0$ we conclude that $\kappa_{12}^{(a)} \geq 0$, and with a similar argument $\kappa_{12}^{(b)} \geq 0$, and $\kappa_{13}^{(a)} \geq 0$.

 Since the Hamiltonian is linear with respect to $f$, the maximum will be attained when $f$ itself is maximized, considering the demand and supply constraints. In particular, it always holds that $f_{13}^{(a)} = \min\{x_1^{(a)},\gamma-x_3^{(a)}\}$, so we shall consider each case separately:
 \begin{enumerate}
     \item[(i)] if  $f_{13}^{(a)} = x_1^{(a)}$, then $f_{12}^{(a)}$ = 0, because the demand of commodity $a$ has already been fully satisfied, and $f_{12}^{(b)} = \min\{x_1^{(b)},\gamma-x_2^{(a)}-x_2^{(b)}\}$. 
     \item[(ii)] if $f_{13}^{(a)} = \gamma - x_3^{(a)}$ and $2\gamma - x_2^{(a)} - x_3^{(a)} \geq x_1^{(a)} + x_1^{(b)}$, i.e., the sum of the supplies is bigger than the sum of the demands, then $f_{12}^{(a)} = x_1^{(a)} -f_{13}^{(a)}$ and $f_{12}^{(b)} = x_1^{(b)}$.
      \item[(iii)] if $f_{13}^{(a)} = \gamma - x_3^{(a)}$ and $2\gamma - x_2^{(a)} - x_3^{(a)} < x_1^{(a)} + x_1^{(b)}$, i.e., the sum of the supplies is smaller than the sum of the demands, then the only non-zero Lagrange multipliers are $\eta_3$ and $\eta_4$. Hence, $\kappa_{12}^{(a)} - \kappa_{12}^{(b)} = \eta_3 - \eta_3 = 0$, which implies that any control such that $f_{12}^{(a)} + f_{12}^{(b)} = \gamma - x_2^{(a)} - x_3^{(a)}$ is optimal. Hence, one optimal strategy is to block one of the commodities.
 \end{enumerate}

\subsection{A Numerical Example}
We conclude this section with an example on a synthetic network. Consider the traffic network in Figure \ref{fig:multi_od} shared by the two commodities $a$ and $b$, where $\mc{O}^{(a)} = \{1\}$, $\mc{O}^{(b)} = \{1,4\}$, $\mc{S}^{(a)} = \{9,10\}$, and $\mc{S}^{(b)} = \{10\}$, and cell $3$ is not accessible by commodity $b$. Moreover, set the demand and supply functions such that $d_i^{(a)}(x_i^{(a)}) = 4x_i^{(a)}$, $d_i^{(b)}(x_i^{(b)}) = 2x_i^{(b)}$, and $s_i(4x_i^{(a)} + 2x_i^{(b)}) = 4 -4x_i^{(a)} - 2x_i^{(b)}$. Furthermore, the exogenous inflows are set such that $\lambda_1^{(a)} = \lambda_1^{(b)} = \lambda_4^{(b)} = 1$, the initial conditions are $x_i^{(k)} = 0.5$ for every $i$ and $k$ and the chosen cost function is $\sum_{i \in \mc{E}}x_i^{(a)} + x_i^{(b)}$. The trajectories of the simulation is shown in Figure~\ref{fig:results}, where it is clear that the optimal control action reduces the cost function. Moreover, in cell $1$, the simulation indicates that it is optimal to block one of the flows, which is not unexpected as the optimality of similar behavior has been observed in part (iii) of the analytical example above.. In fact, at the start, the supply constraint of cell $1$ is met with equality, and the flow of commodity $b$ gets blocked until the constraint is inactive.

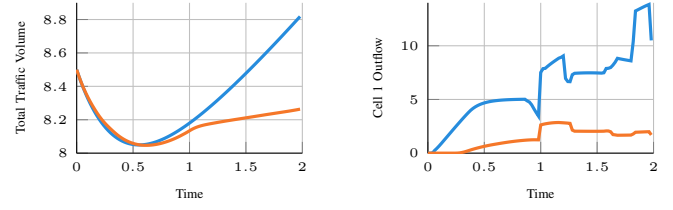
\begin{figure}
	\vspace{0.5cm}
	\centering
	\begin{minipage}{0.45\columnwidth}
		\centering
%
%
\definecolor{mycolor1}{rgb}{0.14900,0.54900,0.86600}%
\definecolor{mycolor2}{rgb}{0.96000,0.46600,0.16000}%
\definecolor{mycolor3}{rgb}{0.85098,0.85098,0.85098}%
\definecolor{mycolor4}{rgb}{0.07059,0.07059,0.07059}%
\begin{tikzpicture}

\begin{axis}[%
width=3cm,
height=2cm,
at={(1.288in,0.657in)},
scale only axis,
xmin=0,
xmax=2,
xlabel={Time},
ymin=8,
ymax=8.9,
xlabel near ticks,
ylabel near ticks,
label style={font=\tiny},
tick label style={font=\tiny},
ylabel={Total Traffic Volume},
axis x line*=bottom,
axis y line*=left,
xmajorgrids,
ymajorgrids
]
\addplot [color=mycolor1, line width=1.2pt, forget plot]
  table[row sep=crcr]{%
0	8.5\\
0.0199999999999996	8.46\\
0.0399999999999991	8.42266666666667\\
0.0600000000000005	8.38782222222222\\
0.0800000000000001	8.35531277343527\\
0.0999999999999996	8.32500381438489\\
0.119999999999999	8.29677654871621\\
0.140000000000001	8.27052489843571\\
0.16	8.24615308452268\\
0.18	8.22357368127149\\
0.199999999999999	8.2027060611054\\
0.220000000000001	8.18347515960401\\
0.24	8.16581050175848\\
0.26	8.14964544015895\\
0.279999999999999	8.1349165640953\\
0.300000000000001	8.12156324558266\\
0.32	8.10952729426895\\
0.34	8.09875269818936\\
0.359999999999999	8.08918543153521\\
0.380000000000001	8.0807733141188\\
0.4	8.07346591014492\\
0.42	8.06721445633197\\
0.44	8.06197181143851\\
0.460000000000001	8.05769242091086\\
0.48	8.05433229172976\\
0.5	8.05184897364809\\
0.52	8.05020154391748\\
0.540000000000001	8.04935059333487\\
0.56	8.04925821202865\\
0.58	8.04988797387426\\
0.6	8.05120491880087\\
0.619999999999999	8.05317553254125\\
0.640000000000001	8.0557677236022\\
0.66	8.05895079740335\\
0.68	8.06269542765942\\
0.699999999999999	8.06697362517314\\
0.720000000000001	8.07175870426939\\
0.74	8.07702524714298\\
0.76	8.08274906641602\\
0.779999999999999	8.08890716621118\\
0.800000000000001	8.09547770204711\\
0.82	8.10243993985396\\
0.84	8.10977421439374\\
0.859999999999999	8.11746188735242\\
0.880000000000001	8.12548530535036\\
0.9	8.13382775809627\\
0.92	8.14247343688707\\
0.960000000000001	8.16061550057089\\
1	8.17980151969387\\
1.04	8.19993340295087\\
1.08	8.22092376348089\\
1.12	8.24269481258517\\
1.16	8.26517733363787\\
1.2	8.28830973860672\\
1.24	8.31203720715308\\
1.28	8.33631090641184\\
1.32	8.36108728817975\\
1.36	8.38632745928659\\
1.4	8.41199662030897\\
1.46	8.45123733334571\\
1.52	8.49128139387615\\
1.58	8.53204956068071\\
1.64	8.5734734139434\\
1.7	8.61549344115423\\
1.76	8.6580574848077\\
1.82	8.7011194890724\\
1.9	8.75923993061086\\
1.98	8.81808775114928\\
};
\addplot [color=mycolor2, line width=1.2pt, forget plot]
  table[row sep=crcr]{%
0	8.5\\
0.0199999999999996	8.46\\
0.0399999999999991	8.42400000010028\\
0.0600000000000005	8.39132800023436\\
0.0800000000000001	8.36144128039142\\
0.0999999999999996	8.33390238777438\\
0.119999999999999	8.30835936130226\\
0.140000000000001	8.28452945989678\\
0.16	8.26218578245518\\
0.18	8.24114627950219\\
0.199999999999999	8.22136586424184\\
0.220000000000001	8.20344639250449\\
0.24	8.18717418865676\\
0.26	8.17237288599021\\
0.279999999999999	8.15868379126948\\
0.300000000000001	8.14556517744705\\
0.32	8.13300851346034\\
0.34	8.12094752478614\\
0.359999999999999	8.10963692322269\\
0.380000000000001	8.09942541010537\\
0.4	8.09026959642713\\
0.42	8.08209853690306\\
0.44	8.07484632728561\\
0.460000000000001	8.06845768716281\\
0.48	8.06299155186561\\
0.5	8.05839421039605\\
0.52	8.05461659666948\\
0.540000000000001	8.05161392196222\\
0.56	8.04934530036135\\
0.58	8.04781408536845\\
0.6	8.04700005036745\\
0.619999999999999	8.04681781774484\\
0.640000000000001	8.04723605184185\\
0.66	8.04822583233051\\
0.68	8.04976039808571\\
0.699999999999999	8.05181492612154\\
0.720000000000001	8.05436634070703\\
0.74	8.05739315218617\\
0.76	8.06087532499483\\
0.779999999999999	8.06479417431559\\
0.800000000000001	8.06913229102582\\
0.82	8.07387350426266\\
0.84	8.07902630767615\\
0.859999999999999	8.0845803455756\\
0.880000000000001	8.09055304234708\\
0.9	8.09694737610219\\
0.92	8.10373389949709\\
0.960000000000001	8.11817141716351\\
0.98	8.12585137969008\\
1.02	8.14191679087227\\
1.04	8.14880294315685\\
1.06	8.15455639222357\\
1.08	8.15938940251289\\
1.1	8.16350387983321\\
1.12	8.16706122454048\\
1.14	8.17018998557775\\
1.18	8.17554754712553\\
1.26	8.18529724606538\\
1.32	8.19220086175562\\
1.4	8.20080735218249\\
1.56	8.21772033055212\\
1.68	8.23072105810046\\
1.88	8.25146699398332\\
1.94	8.25841270434611\\
1.98	8.2633302862193\\
};
\end{axis}

\end{tikzpicture}%
	\end{minipage}
	\hfill
	\begin{minipage}{0.45\columnwidth}
		\centering
%
%
\definecolor{mycolor1}{rgb}{0.14900,0.54900,0.86600}%
\definecolor{mycolor2}{rgb}{0.96000,0.46600,0.16000}%
\definecolor{mycolor3}{rgb}{0.85098,0.85098,0.85098}%
\definecolor{mycolor4}{rgb}{0.07059,0.07059,0.07059}%
\begin{tikzpicture}

\begin{axis}[%
width=3cm,
height=2cm,
at={(1.288in,0.657in)},
scale only axis,
xmin=0,
xmax=2,
xlabel near ticks,
ylabel near ticks,
label style={font=\tiny},
tick label style={font=\tiny},
xlabel={Time},
ymin=0,
ymax=14,
ylabel={Cell 1 Outflow},
axis x line*=bottom,
axis y line*=left,
xmajorgrids,
ymajorgrids,
]
\addplot [color=mycolor1, line width=1.2pt, forget plot]
  table[row sep=crcr]{%
0	0\\
0.0399999999999991	0\\
0.0600000000000005	0.252458236019958\\
0.0800000000000001	0.450460119087772\\
0.0999999999999996	0.673315568304774\\
0.119999999999999	0.910534499296665\\
0.16	1.40288949325513\\
0.220000000000001	2.14343027557288\\
0.26	2.6297131275972\\
0.300000000000001	3.10864722099096\\
0.32	3.34213795383964\\
0.34	3.56971642274915\\
0.359999999999999	3.79050247179895\\
0.380000000000001	3.99698300330625\\
0.4	4.1681049743984\\
0.42	4.31023945950875\\
0.44	4.42851533215472\\
0.460000000000001	4.52708898314973\\
0.48	4.60934778726897\\
0.5	4.67806547082804\\
0.52	4.73552240658236\\
0.540000000000001	4.78359981546735\\
0.56	4.82385400975691\\
0.58	4.85757560657945\\
0.6	4.88583714981407\\
0.619999999999999	4.90953141406734\\
0.640000000000001	4.92940259749874\\
0.66	4.94607188478813\\
0.68	4.9600582394751\\
0.699999999999999	4.97179558692778\\
0.720000000000001	4.98164709382907\\
0.74	4.98991682026586\\
0.76	4.99685945779184\\
0.779999999999999	5.00268851967427\\
0.800000000000001	5.00758298988617\\
0.82	5.01169293669779\\
0.84	5.01514426587324\\
0.859999999999999	5.01804231011326\\
0.880000000000001	4.86981439941081\\
0.9	4.75881277437082\\
0.92	4.49723078120681\\
0.94	4.12657303730373\\
0.960000000000001	3.7797253965409\\
0.98	3.45644291331616\\
1	7.5131446324029\\
1.02	7.8694821538191\\
1.04	7.8968318223291\\
1.06	8.07958604736714\\
1.08	8.24837411077188\\
1.1	8.40636068936998\\
1.12	8.55656014077197\\
1.14	8.70169652868529\\
1.16	8.8441863598541\\
1.18	8.92362545476505\\
1.2	9.05575829284455\\
1.22	6.95382697724918\\
1.24	6.66976244718517\\
1.26	6.65683377972609\\
1.28	7.30668913149818\\
1.3	7.43603039041327\\
1.32	7.448491350036\\
1.34	7.45883957255688\\
1.36	7.46711786837872\\
1.38	7.47336508301842\\
1.4	7.47760932222826\\
1.42	7.47989531378318\\
1.44	7.48023720406234\\
1.46	7.47867404466467\\
1.48	7.47522469381355\\
1.5	7.46991524663834\\
1.52	7.46276977154696\\
1.54	7.45380169437156\\
1.56	7.48818299670325\\
1.58	7.88601044125219\\
1.6	7.90486053280532\\
1.62	7.97652212001336\\
1.64	8.1955948083501\\
1.66	8.47273694750233\\
1.68	8.83147901188629\\
1.7	8.78794913877664\\
1.72	8.74604010964227\\
1.74	8.7062547754255\\
1.76	8.66885767900112\\
1.78	8.6332834371218\\
1.8	8.59579238261576\\
1.82	10.3943595489296\\
1.84	13.2425901210649\\
1.86	13.3470471293951\\
1.9	13.5610103879804\\
1.92	13.6646321099285\\
1.94	13.760835000032\\
1.96	13.8495041462398\\
1.98	10.5021178913419\\
};
\addplot [color=mycolor2, line width=1.2pt, forget plot]
  table[row sep=crcr]{%
0	0\\
0.24	4.22080717044082e-08\\
0.26	0.0064041051053505\\
0.28	0.0230695328774084\\
0.3	0.0572434216100706\\
0.32	0.104996498938104\\
0.34	0.164187041534427\\
0.36	0.233446866001435\\
0.38	0.30147336798674\\
0.4	0.365487456226991\\
0.42	0.425694487732987\\
0.44	0.482303010290612\\
0.46	0.535570383220464\\
0.48	0.586619649757384\\
0.5	0.635238637495276\\
0.52	0.681261115633045\\
0.54	0.724558617634581\\
0.56	0.765033176775995\\
0.58	0.803154673625187\\
0.64	0.91037716792246\\
0.66	0.944849874148542\\
0.68	0.978135469328206\\
0.7	1.01001513774141\\
0.72	1.04032721290909\\
0.74	1.06895698110652\\
0.76	1.09582819653374\\
0.78	1.12089602906198\\
0.8	1.14414120974881\\
0.82	1.16556516267236\\
0.84	1.18518595644367\\
0.86	1.20303493611448\\
0.88	1.21915287740019\\
0.9	1.23152383396209\\
0.92	1.24084825294953\\
0.94	1.2452332518061\\
0.96	1.24373178870643\\
0.98	1.23943830748131\\
1	2.62299057005475\\
1.02	2.68471493370226\\
1.04	2.72894633882497\\
1.06	2.76605323388264\\
1.08	2.79668838912174\\
1.1	2.82082641568823\\
1.12	2.83847229928344\\
1.14	2.84963939357685\\
1.16	2.85433417038735\\
1.18	2.84079105201137\\
1.2	2.82821375950339\\
1.22	2.81163885994756\\
1.24	2.79572870117684\\
1.26	2.77277944174737\\
1.28	2.14771968335602\\
1.3	2.04626721085072\\
1.38	2.03921059628042\\
1.42	2.03661683176588\\
1.44	2.03579581713926\\
1.46	2.03538917505579\\
1.48	2.03547070699618\\
1.5	2.03611761094612\\
1.52	2.03741019234861\\
1.54	2.0394309160762\\
1.56	2.04465918033129\\
1.58	2.06905638896377\\
1.6	2.06589190414779\\
1.62	1.99080769250405\\
1.64	1.73057374187026\\
1.66	1.7023260491438\\
1.68	1.66821479325093\\
1.72	1.67270884564604\\
1.74	1.67533163061861\\
1.76	1.67838648082779\\
1.8	1.68490841908761\\
1.82	1.78719580373764\\
1.84	1.9473380756804\\
1.86	1.95569125974336\\
1.92	1.98301555975763\\
1.94	1.99080187829263\\
1.96	1.99751248688611\\
1.98	1.69753053587215\\
};
\end{axis}

\end{tikzpicture}%
	\end{minipage}
	\caption{Left: Total traffic volume over time, in blue uncontrolled dynamics and in red the optimal control one. Right: Cell 2 Outflow over time: in blue commodity $a$ and in red commodity $b$.}
	\label{fig:results}
\end{figure}

\section{Conclusions}
In this paper we leveraged a dynamical multi-commodity flow network model to formulate a tight convex relaxation of the Multi-Commodity Dynamic Traffic Assignment problem, which can be solved using available convex optimization solvers. Future research should aim at developing a distributed algorithm to solve the problem and also account for implementation issues such as discretized control actions. 
\bibliographystyle{IEEEtran}
\bibliography{references}
		
\end{document}